\documentclass[11pt]{article}

\usepackage[british]{babel}
\usepackage[utf8]{inputenc}
\usepackage{amsmath,amsthm,amsfonts,amssymb}
\usepackage{multicol}
\usepackage[all]{xy}
\usepackage{enumerate}
\usepackage{wrapfig}
\usepackage{geometry}
\usepackage{hyperref}

\newtheorem{theorem}{Theorem}[section]
\newtheorem{proposition}[theorem]{Proposition}
\newtheorem{lemma}[theorem]{Lemma}
\newtheorem{corollary}[theorem]{Corollary}
\theoremstyle{definition}
\newtheorem{definition}[theorem]{Definition}
\newtheorem{remark}[theorem]{Remark}

\newcommand{\Coin}{\operatorname{Coin}}
\newcommand{\Fix}{\operatorname{Fix}}

\newcommand{\Hom}{\operatorname{Hom}}
\newcommand{\Id}{\operatorname{Id}}
\newcommand{\trace}{\operatorname{trace}}

\title{A Lefschetz type homomorphism for coincidence of several maps}
\author{PRYSCILLA DOS SANTOS FERREIRA SILVA\thanks{Departamento de Ci\^encias Exatas, Universidade Estadual de Santa Cruz, Rodovia Jorge Amado, Km 16, Bairro Salobrinho, CEP 45662-900, Ilh\'eus-BA, Brazil. e-mail: psfsilva@uesc.br}
\and
WESLEM LIBERATO SILVA\thanks{Departamento de Ci\^encias Exatas, Universidade Estadual de Santa Cruz, Rodovia Jorge Amado, Km 16, Bairro Salobrinho, CEP 45662-900, Ilh\'eus-BA, Brazil. e-mail: wlsilva@uesc.br}}
\date{}

\begin{document}
\maketitle

\begin{abstract}
Given $p$ maps $f_1,\cdots,f_p:X\to M,$ $p\geq 2,$ from an arbitrary topological space to an oriented closed connected $m$-manifold, we define a graded homomorphism $\Lambda_{f_1\cdots f_p}:H(X)\to H(M^{p-1})$ of degree $-m(p-1)$ called the Lefschetz homomorphism. If the Lefschetz homomorphism is nontrivial, then there is a point $x\in X$ such that $f_1(x)=\cdots=f_p(x).$ The Lefschetz homomorphism $\Lambda_{f_1\cdots f_p}$ can be represented as a Knill-like trace.
\end{abstract}


\section{Introduction}

Let $f,g:X\to M$ be two maps from a topological space $X$ into an oriented compact $m$-manifold $M.$ Saveliev \cite{Saveliev} developed a Lefschetz coincidence theory to study the set $\Coin(f,g)=\{x\in X\mid f(x)=g(x)\}.$ The author defined a graded homomorphism $\Lambda_{fg}:H(X)\to H(M)$ of degree $(-m)$ and proved that if $\Lambda_{fg}$ is nontrivial then there exists a point $x\in X$ such that $f(x)=g(x).$

A natural generalization is the coincidence problem for several maps: given $f_1,f_2,\cdots,f_p:X\to M$ is there $x\in X$ such that $f_1(x)=\cdots=f_p(x)$ ? Several authors have studied this problem; see, for example, \cite{BLM,Staecker,MSilva,MSpiez,MWong}.

In \cite{BLM} the authors defined a cohomology class $L(f_1,\cdots,f_p)$ called the Lefschetz class, such that $L(f_1,\cdots,f_p)\neq 0$ implies $\Coin(f_1,\cdots,f_p)\neq\emptyset.$ This theory was generalized in \cite{MSpiez} when $M$ is not necessarily orientable.

In \cite{MWong}, obstruction theory was used to study when $(f_1,\cdots,f_p)$ can be deformed to $(f'_1,\cdots,f'_p)$ such that $\Coin(f'_1,\cdots,f'_p)=\emptyset.$ This theory was generalized for fiber bundles in \cite{MSilva}.

In this work we study $\Coin(f_1,\cdots,f_p)$ using a homological approach and extend the theory developed in \cite{Saveliev} to $p$-maps $f_1,\cdots,f_p:X\to M,$ where $p\geq 2.$ We define a graded homomorphism $\Lambda_{f_1\cdots f_p}:H(X)\to H(M^{p-1}),$ called the Lefschetz homomorphism, and prove that if $\Lambda_{f_1\cdots f_p}$ is nontrivial, then the set $\Coin(f_1,\cdots,f_p)$ is not empty. We prove in Theorem \ref{mainfactor} that
\[
\Lambda_{f_1\cdots f_p}
=\pm\eta'\circ(\Lambda_{f_1f_2}\otimes\cdots\otimes\Lambda_{f_1f_p})\circ\kappa_X\circ\Delta_X^{(p-1)},
\]
where $\Delta_X^{(p-1)}$ is the homological coproduct induced by the $(p-1)$-fold diagonal and $\kappa_X$ is the sign homomorphism defined in Section 4.

When $p=2,$ $X=M\times Y$, and $f_1$ or $f_2$ is the projection map $\pi:M\times Y\to M$, Knill \cite{Knill} defined the Knill trace to study fixed points of parametrized maps. The Lefschetz homomorphism $\Lambda_{f_1\cdots f_p}$ can be represented as a Knill-like trace, see Remark \ref{knilllike}.

The paper is organized as follows. In Section 2 we present the definitions of the Knill trace for a parametrized map and the Lefschetz and coincidence homomorphisms. In Section 3 we define the coincidence homomorphism for $p$-maps and present a representation formula, Theorem \ref{representation}. In Section 4 we define the Lefschetz homomorphism for $p$-maps and relate it to the homomorphisms for coincidences of two maps; see Theorem \ref{mainfactor}.
In Section 5 we compute the Lefschetz homomorphism for some special cases.

\section{Preliminaries}

\subsection{Notations}

Given an $m$-manifold $M$ we will denote
\[
M^{\times}=(M\times M,M\times M\setminus\Delta_M),
\]
where $\Delta_M=\{(x,x)\mid x\in M\}$ and
\[
M^{\times p}=(M^{\times})^p=M^{\times}\times\cdots\times M^{\times},\quad (p\text{-times}).
\]

We will also denote by $\Delta_{M^p}=\{((z_1,\cdots,z_p),(z_1,\cdots,z_p))\mid z_i\in M\}$ the diagonal in $M^p\times M^p$ and
\[
M^{\times}_p=(M^p\times M^p,M^p\times M^p\setminus\Delta_{M^p}).
\]

We recall that the product of two pairs of topological spaces $(Z_1,B_1)$ and $(Z_2,B_2)$ is defined by
\[
(Z_1,B_1)\times(Z_2,B_2):=(Z_1\times Z_2,Z_1\times B_2\cup B_1\times Z_2).
\]
More generally, given $(Z_i,B_i),$ $i=1,\cdots,k$, $k$ pairs of topological spaces, we define
\[
\prod_{i=1}^k(Z_i,B_i):=\left(\prod_{i=1}^kZ_i,\,\,\bigcup_{i=1}^kY_i\right),
\]
where $Y_i=Z_1\times\cdots\times Z_{i-1}\times B_i\times Z_{i+1}\times\cdots\times Z_k.$

We use the sign conventions of \cite{Saveliev} for the cross, cup and cap products, namely those of Dold.

\begin{lemma}\label{cap-product}
Let $X_i$ be topological spaces, $u_i\in H^{r_i}(X_i,R)$ and $z_i\in H_{s_i}(X_i,R)$ for each $i=1,\cdots,n.$ Then
\[
(u_1\times u_2\times\cdots\times u_n)\smallfrown(z_1\times z_2\times\cdots\times z_n)
=(-1)^{A_n}(u_1\smallfrown z_1)\times\cdots\times(u_n\smallfrown z_n),
\]
where
\[
A_n=\sum_{j=2}^n r_j\left(\sum_{i=1}^{j-1}s_i\right).
\]
\end{lemma}

\begin{proof}
For two factors, the cap-product formula with the chosen convention is
\[
(u_1\times u_2)\smallfrown(z_1\times z_2)
=(-1)^{r_2s_1}(u_1\smallfrown z_1)\times(u_2\smallfrown z_2).
\]
The general case follows by induction.
\end{proof}

\begin{corollary}\label{sign-cor}
If $r_i=r$ and $s_i=s$ for all $i=1,\cdots,n$ then
\[
A_n=\frac{n(n-1)rs}{2}
\]
\end{corollary}

If $\pi_1:M\times M\to M$ is the projection on the first coordinate then the Thom isomorphism $\varphi_M:H_k(M^{\times})\to H_{k-m}(M)$ is given by $\varphi_M(x)={\pi_1}_{\#}(\tau_M\smallfrown x),$ where $\tau_M\in H^m(M^{\times})$ is the Thom class of the fiber bundle pair $\xi_2(M)\equiv M^{\times}\stackrel{{\scriptstyle\pi_1}}{{\longrightarrow}}M.$

The $p$-fold product of $\xi_2(M)$ yields the fiber bundle pair $\xi_p(M)\equiv M^{\times p}\stackrel{{\scriptstyle\pi}}{{\longrightarrow}}M^p$ where $\pi=\underbrace{\pi_1\times\cdots\times\pi_1}_{p\text{-times}}.$ The Thom class of $\xi_p(M)$ will be denoted by $\tau_{M^p}.$

The Thom class of the fiber bundle pair $M^{\times}_p\stackrel{{\scriptstyle\pi'}}{{\longrightarrow}}M^p$ will be denoted by $\tau_{M^{\times}_p},$ where $\pi'$ is the projection on the first coordinate.

\subsection{Traces}

The theory of the Lefschetz class of a graded homomorphism presented in \cite[Section 1]{GNO} and \cite[Section 2]{Saveliev} will be useful in this work.

Let $E_{\ast}$ be a finite-dimensional graded vector space over a field $\mathbb{F}.$ Consider the dual graded vector space $E^{\ast},$ where $E^j=\Hom_{\mathbb{F}}(E_j,\mathbb{F}).$ Let $C_{\ast}$ be another graded vector space over $\mathbb{F}$ and suppose that a cap product with values in $C_{\ast}$ is given, that is, a collection of homomorphisms:
\[
E^r\otimes E_s\stackrel{{\scriptstyle\smallfrown}}{{\longrightarrow}}C_{s-r}.
\]
There is a natural isomorphism $\Theta:E^k\otimes E_{k+n}\to\Hom_{\mathbb{F}}(E_k,E_{k+n})$ given by
\[
\Theta(h\otimes y)(x)=(-1)^{|y||x|}\langle h,x\rangle y,
\]
where $|w|$ denotes the degree of $w,$ $\langle,\rangle:E^j\otimes E_j\to\mathbb{F}$ is the dual pairing, $h\in E^k,$ $x\in E_k$ and $y\in E_{k+n}.$

\begin{remark}
There is a difference in the definition of $\Theta$ compared with the one given in \cite[Section 1]{GNO}. That reference defines $\Theta(h\otimes y)(x)=\langle h,x\rangle y,$ that is, the definitions differ by a sign. In this work, to present the definition of Lefschetz homomorphism for the $p$-maps $f_1,\cdots,f_p$, we will use the definition above.
\end{remark}

\begin{definition}
The trace of a homomorphism $f\in\Hom_{\mathbb{F}}(E_k,E_{k+n})$ of degree $n$ is given by
\[
\trace(f)=\smallfrown\Theta^{-1}(f)\in C_n.
\]
\end{definition}

\begin{proposition}
With the above sign convention, the argument of \cite[Proposition 1.2]{GNO} gives the following. If $\{a_j^k\mid j=1,\dots,\alpha_k\}$ is a basis for $E_k$ with corresponding dual basis $\{\bar a_j^k\mid j=1,\dots,\alpha_k\}$ of $E^k$ then
\[
\trace(f)=(-1)^{k(k+n)}\sum_{j=1}^{\alpha_k}\bar a_j^k\smallfrown f(a_j^k).
\]
\end{proposition}

\begin{definition}
Let $E_{\ast}$ be a finite-dimensional graded vector space equipped with a cap product taking values in the graded vector space $C_{\ast}.$ If $h:E_{\ast}\to E_{\ast}$ is a homomorphism of degree $n$ then the Lefschetz class of $h,$ $L_n(h)\in C_n,$ is:
\[
L_n(h)=\sum_k\trace(h_k).
\]
\end{definition}

As in \cite[Definition 2.1 and Proposition 2.2]{Saveliev}, we obtain the following expression of the Lefschetz class $L_n(h)$ as a Knill-like trace:

\begin{proposition}\label{knill-expression}
If $h:E_{\ast}\to E_{\ast}$ is a homomorphism of degree $n$ then
\[
L_n(h)=\sum_{k\geq 0}(-1)^{k(k+n)}\sum_{j=1}^{\alpha_k}\bar a_j^k\smallfrown h(a_j^k)
\]
where $\{a_j^k\mid j=1,\dots,\alpha_k\}$ is a basis for $E_k$ with corresponding dual basis $\{\bar a_j^k\mid j=1,\dots,\alpha_k\}$ of $E^k,$ for each $k.$
\end{proposition}

We have analogous definitions considering the dual graded vector space $E^{\ast}$ see \cite[Definitions 1.5 and 1.7]{GNO}.

\subsection{The Knill trace for a parametrized map}

Let $M$ be an oriented connected compact manifold without boundary. Let $F:Y\times M\to M$ be a continuous parametrized map, where $Y$ is the parameter space and $\mathbb{F}$ is a field. We say that $(y,x)$ is a fixed point of $F$ if $F(y,x)=x.$

Let $u\in H_n(Y)$ and $F_u:H(M)\to H(M)$ be the homomorphism of degree $n$ given by $F_u(v)=F_{\ast}(u\times v).$ Let $H^j(M)\otimes H_{j+n}(M)\stackrel{{\scriptstyle\smallfrown}}{{\to}}H_n(M)$ be the cap product. Following \cite{Knill}, we define the Lefschetz class of $F,$ $L_n(F_u)\in H_n(M).$

\begin{definition}
The Knill trace of $F:Y\times M\to M$ is the homomorphism of degree 0 $L(F):H_{\ast}(Y)\to H_{\ast}(M)$ given by $L(F)(u)=L_n(F_u)\in H_n(M),$ where $u\in H_n(Y).$
\end{definition}

It follows from \cite[Theorem 1]{Knill} that if $L(F)$ is nontrivial then $F$ has a fixed point. Note that $L(F)$ depends only on the homotopy class of $F,$ and from Proposition \ref{knill-expression} we have:
\[
L(F)(u)=\sum_{k\geq 0}(-1)^{k(k+n)}\sum_{j=1}^{\beta_k}\bar b_j^k\smallfrown F_{\ast}(u\times b_j^k)
\]
where $u\in H_n(Y)$ and $\{b_j^k\mid j=1,\dots,\beta_k\}$ is a basis for $H_k(M)$ with corresponding dual basis $\{\bar b_j^k\mid j=1,\dots,\beta_k\}$ of $H^k(M),$ for each $k.$ For more details and some calculations of this trace see \cite{GNO}.

\subsection{The Lefschetz and coincidence homomorphisms}

Most of this subsection follows \cite{Saveliev}. From now on the homology and cohomology groups will be with rational coefficients.

\begin{remark}
The theory developed in \cite{Saveliev} considers the manifold $M$ possibly with boundary. Thus, in this subsection only, we allow $M$ to have boundary.
\end{remark}

Let $f:(X,A)\to(M,\partial M)$ and $g:X\to M$ be two maps, where $X$ is a topological space, with $A\subset X,$ and $M$ is an oriented connected compact $m$-manifold with boundary $\partial M.$

Following \cite[Section 3]{Saveliev}, attach an exterior collar to the boundary and put
\[
M'=M\cup_{\partial M}(\partial M\times[0,1)).
\]
In this subsection, when $M$ has boundary, we use
\[
M^{\times}=(M\times M',\,M\times M'\setminus\delta(M')),
\qquad \delta(M')=\{(x,x)\mid x\in M'\}.
\]
For $\partial M=\emptyset$, this agrees with the notation of Subsection 2.1.

We suppose that $\Coin(f,g)\cap A=\emptyset.$ The map $(f\times g)\circ\delta:(X,A)\to M^{\times}$ is well defined.

\begin{definition}
The coincidence homomorphism of the pair $(f,g)$ is the homomorphism $I_{fg}:H(X,A)\to H(M^{\times})$ of degree $0$ defined by
\[
I_{fg}=(f\times g)_{\#}\circ\delta_{\#}.
\]
\end{definition}

It is immediate that $I_{fg}\neq 0\Longrightarrow\Coin(f,g)\neq\emptyset.$ There is a representation formula for $I_{fg}$ in terms of the fundamental class $O_M$ of $M,$ see \cite[Theorem 4.5]{Saveliev}.

We consider the transfer of $f$ with respect to a fixed element $z\in H_{n+m}(X,A),$ namely, the homomorphism $f_!^z:H_{\ast}(M)\to H_{\ast+n}(X)$ of degree $n$ defined by:
\[
f_!^z=(f^{\#}\circ D^{-1})\smallfrown z,
\]
where $D:H^{\ast}(M,\partial M)\to H_{m-\ast}(M)$ is the Poincar\'e-Lefschetz duality isomorphism with
\[
D(x)=x\smallfrown O_M,
\]
and $O_M\in H_m(M,\partial M)$ is the fundamental class of $(M,\partial M).$ Therefore the homomorphism $g_{\#}\circ f_!^z:H(M)\to H(M)$ has degree $n.$

\begin{definition}
The Lefschetz homomorphism $\Lambda_{fg}:H_{\ast}(X,A)\to H_{\ast-m}(M)$ of the pair $(f,g)$ is the homomorphism of degree $(-m)$ given by
\[
\Lambda_{fg}(z)=L(g_{\#}\circ f_!^z),\qquad z\in H_{n+m}(X,A).
\]
\end{definition}

Note that the degree of $g_{\#}\circ f_!^z$ is $|z|-m.$ Thus we can represent the Lefschetz homomorphism $\Lambda_{fg}$ as a Knill-like trace by
\[
\Lambda_{fg}(z)=\sum_k(-1)^{k(k+|z|-m)}\sum_j\bar a_j^k\smallfrown g_{\#}\circ f_!^z(a_j^k),
\]
where $\{a_j^k\mid j=1,\dots,\alpha_k\}$ is a basis for $H_k(M)$ with corresponding dual basis $\{\bar a_j^k\mid j=1,\dots,\alpha_k\}$ of $H^k(M),$ for each $k.$ It follows from \cite[Theorem 6.1]{Saveliev} that if $\Lambda_{fg}(z)\neq 0$ then the pair $(f,g)$ has a coincidence. The next result gives a relation between the coincidence homomorphism and the Lefschetz homomorphism.

\begin{theorem}\label{saveliev-thm}
\cite[Theorem 6.1]{Saveliev} The coincidence homomorphism is equal to the Lefschetz homomorphism:
\[
\varphi_M\circ I_{fg}=\Lambda_{fg}.
\]
Moreover, if $\Lambda_{fg}\neq 0$ then $(f,g)$ has a coincidence.
\end{theorem}

The Lefschetz homomorphism for coincidence generalizes the Knill trace of a parametrized map defined in \cite[Definition 2.1]{GNO}.

\begin{corollary}
\cite[Corollary 5.7]{Saveliev} Suppose $\partial M=\emptyset$. Let $F,p:Y\times M\to M$ be maps where $p$ is the projection. We have $\Fix(F)=\Coin(p,F)$ and
\[
\Lambda_{pF}(u\times O_M)=L(\widetilde F_u),\qquad u\in H(Y),
\]
where $\widetilde F_u:H(M)\to H(M)$ is given by $\widetilde F_u(x)=(-1)^{(m-|x|)|u|}F_{\#}(u\times x).$
\end{corollary}

\section{The coincidence homomorphism for $p$-maps}

Let $f_1,\cdots,f_p:X\to M$ be continuous maps, where $p\geq 2,$ $X$ is a topological space and $M$ is an oriented connected compact $m$-manifold without boundary. We define $F,G:X\to M^{p-1},$ the $(p-1)$-fold product of $M,$ by
\[
F=(f_1,\cdots,f_1)\qquad\text{and}\qquad G=(f_2,\cdots,f_p).
\]
Note that $\Coin(F,G)=\Coin(f_1,\cdots,f_p).$

\begin{proposition}
The map $f:X\to M^p$ given by $f=(f_1,\cdots,f_p)$ can be deformed to a map $g=(g_1,\cdots,g_p)$ such that $\Coin(g_1,\cdots,g_p)=\emptyset$ if and only if the pair $(F,G)$ can be deformed to a pair of maps $(F',G')$ such that $\Coin(F',G')=\emptyset.$
\end{proposition}

\begin{proof}
This follows from \cite[Proposition 2.2]{MSilva} in the special case where $B$ is a point.
\end{proof}

Proposition 3.1 motivates defining a Lefschetz homomorphism for $(F,G)$ to study the coincidence set of the $p$-maps $f_1,\cdots,f_p.$

Let $A\subset X$ such that $\Coin(F,G)\cap A=\emptyset.$ The map $(F\times G)\circ\delta:(X,A)\to M^{\times}_{p-1}$ is well defined, where $\delta:X\to X\times X$ is given by $\delta(x)=(x,x).$

\begin{definition}
The coincidence homomorphism of the $p$-maps $(f_1,\cdots,f_p),$ is the homomorphism $I_{FG}:H(X,A)\to H(M^{\times}_{p-1})$ of degree $0$ defined by
\[
I_{FG}=(F\times G)_{\#}\circ\delta_{\#}.
\]
We will denote $I_{f_1\cdots f_p}=I_{FG}.$
\end{definition}

It is immediate that $I_{FG}\neq 0\Longrightarrow\Coin(F,G)\neq\emptyset\Longrightarrow\Coin(f_1,\cdots,f_p)\neq\emptyset.$

\begin{proposition}
For dimensional reasons, if $z\in H_j(X,A)$ for $j<m(p-1)$ or $j>2m(p-1)$ we have $I_{FG}(z)=0.$
\end{proposition}

\begin{proposition}\label{psi-homeo}
The map $\psi:M^{\times(p-1)}\to M^{\times}_{p-1}$ defined by:
\[
\psi((x_1,x_2),(x_3,x_4),\cdots,(x_{2(p-1)-1},x_{2(p-1)}))
=((x_1,x_3,\cdots,x_{2(p-1)-1}),(x_2,x_4,\cdots,x_{2(p-1)}))
\]
is a homeomorphism.
\end{proposition}

\begin{proof}
Recall that
\[
M^{\times(p-1)}=\prod_{i=1}^{p-1}(M\times M,M\times M\setminus\Delta_M)
=\left(\prod_{i=1}^{p-1}M^2,\,\,\bigcup_{i=1}^{p-1}Y_i\right),
\]
where $Y_i=M^2\times\cdots\times M^2\times\underbrace{(M\times M\setminus\Delta_M)}_{i\text{-position}}\times M^2\times\cdots\times M^2.$ It is now clear that $\psi$ is well defined. The inverse $\psi^{-1}:M^{\times}_{p-1}\to M^{\times(p-1)}$ is naturally given by
\[
\psi^{-1}((z_1,\cdots,z_{p-1}),(w_1,\cdots,w_{p-1}))=((z_1,w_1),\cdots,(z_{p-1},w_{p-1})).
\]
\end{proof}

The map $\psi$ can be viewed as a product of projections. In fact, denote by $\pi_1^i:M^2=M\times M\to M$ and $\pi_2^i:M\times M\to M$ the projections on the first and the second coordinate, respectively, where $M^2$ is the $i$th factor in $M^{\times(p-1)},$ $i=1,\cdots,p-1.$ Thus
\[
\psi=\widetilde{\pi_1}\times\widetilde{\pi_2}
\]
where $\widetilde{\pi_1},\widetilde{\pi_2}:\prod_{i=1}^{p-1}M^2\to M^{p-1}$ are given by
\[
\widetilde{\pi_1}=\pi_1^1\times\pi_1^2\times\cdots\times\pi_1^{p-1}
\qquad\text{and}\qquad
\widetilde{\pi_2}=\pi_2^1\times\pi_2^2\times\cdots\times\pi_2^{p-1}.
\]

For the next result we consider $A=\emptyset.$ Let $\delta_{p-1}:X\to X^{p-1}$ be the $(p-1)$-fold diagonal map and let
\[
\eta_X:H(X)\otimes\cdots\otimes H(X)\longrightarrow H(X^{p-1})
\]
be the K$\ddot{u}$nneth isomorphism. We define
\[
\Delta_X^{(p-1)}:=\eta_X^{-1}\circ(\delta_{p-1})_{\#}:H(X)\longrightarrow H(X)^{\otimes(p-1)}.
\]
Also, let
\[
\eta:H(M^{\times})\otimes\cdots\otimes H(M^{\times})\longrightarrow H(M^{\times(p-1)})
\]
be the K$\ddot{u}$nneth isomorphism.
We have the following commutative diagram:
\[
\xymatrix@C=1.2em{
H(X) \ar[r]^-{\Delta_X^{(p-1)}} \ar[d]_-{I_{FG}} &  
H(X)^{\otimes(p-1)} \ar[rrr]^-{I_{f_1f_2}\otimes\cdots\otimes I_{f_1f_p}} & & &
H(M^{\times})^{\otimes(p-1)} \ar[r]^-{\eta} &
H(M^{\times(p-1)}) \ar[d]^-{\psi_{\#}} \\
H(M^{\times}_{p-1}) \ar@{=}[rrrrr] & & & & & H(M^{\times}_{p-1}).
}
\tag{2}
\]

\begin{theorem}\label{factor-I}
From Diagram (2) we obtain
\[
\psi_{\#}\circ\eta\circ(I_{f_1f_2}\otimes\cdots\otimes I_{f_1f_p})\circ\Delta_X^{(p-1)}=I_{FG}.
\]
\end{theorem}

\begin{remark}
There is another way to define a type of coincidence homomorphism for the $p$-maps $f_1,\cdots,f_p.$ Consider the homomorphism
\[
\overline{I_{f_1\cdots f_p}}:=(f_1\times\cdots\times f_p)_{\#}\circ(\delta_p)_{\#}:H(X,A)\to H(M^p,M^p\setminus\Delta_p),
\]
where $\delta_p:X\to X^p$ is the $p$-fold diagonal and $\Delta_p=\{(x,\cdots,x)\mid x\in M\}.$ This formulation does not directly use the theory developed in \cite{Saveliev}. Nevertheless, the map $e:(M^p,M^p\setminus\Delta_p)\to M^{\times(p-1)}$ defined by
\[
e(x_1,\cdots,x_p)=((x_1,x_2),(x_1,x_3),\cdots,(x_1,x_p)),
\]
gives the following relation:
\[
\psi_{\#}\circ e_{\#}\circ\overline{I_{f_1\cdots f_p}}=I_{FG}.
\]
\end{remark}

From now on, in the definitions of the transfer and the Lefschetz homomorphism, we take $A=\emptyset.$ Following \cite[Definition 4.4]{Saveliev}, we consider the transfer of $F.$

\begin{definition}
The transfer of $F$ with respect to $z\in H_{m(p-1)+n}(X)$ is the homomorphism $F_!^z:H_{\ast}(M^{p-1})\to H_{\ast+n}(X)$ of degree $n$ given by
\[
F_!^z=(F^{\#}\circ D^{-1})\smallfrown z,
\]
where $D:H^k(M^{p-1})\to H_{m(p-1)-k}(M^{p-1})$ is the Poincar\'e duality isomorphism given by $D(x)=x\smallfrown O_{M^{p-1}},$ and $\smallfrown$ is the usual cap-product.
\end{definition}

By \cite[Theorem 4.5]{Saveliev} we have the following:

\begin{theorem}[Representation formula]\label{representation}
\[
I_{FG}(z)=I_{\#}(\Id\otimes G_{\#}\circ F_!^z)\delta_{\#}(O_{M^{p-1}}),
\]
where $O_{M^{p-1}}$ is the fundamental class of $M^{p-1},$ $I:M^{p-1}\times M^{p-1}\to M^{\times}_{p-1}$ is the inclusion and $G_{\#}\circ F_!^z:H(M^{p-1})\to H(M^{p-1})$ is the homomorphism of degree $n.$
\end{theorem}

\section{The Lefschetz homomorphism for $p$-maps}

For $z\in H_{m(p-1)+n}(X)$ the homomorphism $G_{\#}\circ F_!^z:H_{\ast}(M^{p-1})\to H_{\ast+n}(M^{p-1})$ has degree $n.$ Thus, by Definition 2.6, we obtain $L(G_{\#}\circ F_!^z)\in H_n(M^{p-1}).$

\begin{definition}
The Lefschetz homomorphism for the $p$-maps $f_1,\cdots,f_p:X\to M$ is the homomorphism $\Lambda_{FG}:H_{\ast}(X)\to H_{\ast-m(p-1)}(M^{p-1})$ of degree $-m(p-1)$ given by:
\[
\Lambda_{FG}(z)=L(G_{\#}\circ F_!^z),
\]
where $z\in H_{m(p-1)+n}(X).$ We will denote $\Lambda_{f_1\cdots f_p}=\Lambda_{FG}.$
\end{definition}

\begin{remark}\label{knilllike}
By Proposition \ref{knill-expression} the Lefschetz homomorphism $\Lambda_{FG}$ has a representation as a Knill-like trace given by:
\[
\Lambda_{FG}(z)=\sum_k(-1)^{k(k+|z|-m(p-1))}\sum_j\bar a_j^k\smallfrown G_{\#}\circ F_!^z(a_j^k),
\]
where $\{a_j^k\mid j=1,\dots,\alpha_k\}$ is a basis for $H_k(M^{p-1})$ with corresponding dual basis $\{\bar a_j^k\mid j=1,\dots,\alpha_k\}$ of $H^k(M^{p-1}),$ for each $k.$
\end{remark}

Let $\tau_{M^{\times}_{p-1}}$ be the Thom class of the fiber bundle pair $M^{\times}_{p-1}\stackrel{{\scriptstyle\pi'}}{{\longrightarrow}}M^{p-1}.$ From \cite[Theorem 6.1]{Saveliev} we have:

\begin{theorem}[Lefschetz-type coincidence theorem for $p$-maps]\label{lefschetz-p}
\[
\varphi_{M^{\times}_{p-1}}\circ I_{FG}=\Lambda_{FG},
\]
where $\varphi_{M^{\times}_{p-1}}:H(M^{\times}_{p-1})\to H(M^{p-1})$ is the Thom isomorphism given by
\[
\varphi_{M^{\times}_{p-1}}(w)={\pi'}_{\#}(\tau_{M^{\times}_{p-1}}\smallfrown w).
\]
Moreover, if $\Lambda_{FG}\neq 0$ then $\Coin(f_1,\cdots,f_p)\neq\emptyset.$
\end{theorem}

Given the fiber bundle pair $\xi=M^{\times}\stackrel{{\scriptstyle\pi_1}}{{\longrightarrow}}M$ we can consider the $q$-product, $q\geq 1,$ of that fiber bundle pair with itself: $\xi_q=M^{\times q}\stackrel{{\scriptstyle\pi}}{{\longrightarrow}}M^q,$ where the map $\pi$ is given by $\pi=\underbrace{\pi_1\times\cdots\times\pi_1}_{q\text{-times}}.$

Since $M$ is orientable, choose an orientation for $M,$ with corresponding fundamental class $O_M,$ so that the Thom class $\tau_M$ for $\xi$ is uniquely determined; see \cite{Spanier}. Similarly, $\tau_{M^{\times}_q}$ is unique.

\begin{proposition}\label{thom-product}
Let $\tau_{M^q}$ be the Thom class of the fiber bundle pair $M^{\times q}\stackrel{{\scriptstyle\pi}}{{\longrightarrow}}M^q,$ and let $\psi_q:M^{\times q}\to M^{\times}_q$ be the homeomorphism defined as in Proposition \ref{psi-homeo}. Then
\[
\psi_q^{\#}(\tau_{M^{\times}_q})=\tau_{M^q}=\pm\underbrace{\tau_M\times\cdots\times\tau_M}_{q\text{-times}}.
\]
\end{proposition}

\begin{proof}
The first equality follows from the fact that $\psi_q$ is an isomorphism of orientable fiber bundles and satisfies $\pi'\circ\psi_q=\pi.$ The second equality follows from the characterization of the Thom class and from the product orientation. In fact, the restriction of $\tau_M\times\cdots\times\tau_M$ to every fiber is the product of the orientation generators, and therefore it is the Thom class of the product bundle up to the choice of orientation.
\end{proof}

In the tensor product expressions for the factorization, homomorphisms act componentwise. We define $\overline{\varphi}:H(M^{\times})\otimes\cdots\otimes H(M^{\times})\to H(M)\otimes\cdots\otimes H(M)$, with $(p-1)$ factors, by:
\[
\overline{\varphi}=\varphi_M\otimes\cdots\otimes\varphi_M.
\]
Note that $\overline{\varphi}$ has an inverse given by $\overline{\varphi}^{-1}=\varphi_M^{-1}\otimes\cdots\otimes\varphi_M^{-1}.$

Let $r=p-1.$ For a homogeneous element $x=x_1\otimes\cdots\otimes x_r\in H(M^{\times})^{\otimes r}$, where $|x_i|=s_i,$ we define
\[
\kappa_M(x)=(-1)^{A_r(x)}x,
\qquad
A_r(x)=\sum_{j=2}^{r}m\left(\sum_{i=1}^{j-1}s_i\right).
\]
Analogously, for a homogeneous element $z=z_1\otimes\cdots\otimes z_r\in H(X)^{\otimes r}$, we define
\[
\kappa_X(z)=(-1)^{A_r(z)}z,
\qquad
A_r(z)=\sum_{j=2}^{r}m\left(\sum_{i=1}^{j-1}|z_i|\right).
\]

\begin{proposition}\label{thom-kunneth}
The following relation holds up to sign:
\[
\varphi_{M^r}\circ\eta
=\pm\eta'\circ\overline{\varphi}\circ\kappa_M,
\]
where $\eta'$ is also a K$\ddot{u}$nneth isomorphism. Moreover,
\[
\varphi_{M^r}=\varphi_{M^{\times}_r}\circ\psi_{r\#}.
\]
\end{proposition}

\begin{proof}
The map $\psi_r$ satisfies $\pi'\circ\psi_r=\pi_1\times\cdots\times\pi_1=\pi.$ From Proposition \ref{thom-product} and by naturality of the cap product we have
\[
\begin{array}{lll}
\varphi_{M^r}(w)
&=&\pi_{\#}(\tau_{M^r}\smallfrown w)\\
&=&{\pi'}_{\#}(\psi_{r\#}(\psi_r^{\#}(\tau_{M^{\times}_r})\smallfrown w))\\
&=&{\pi'}_{\#}(\tau_{M^{\times}_r}\smallfrown\psi_{r\#}(w))\\
&=&\varphi_{M^{\times}_r}(\psi_{r\#}(w)).
\end{array}
\]
Therefore $\varphi_{M^r}=\varphi_{M^{\times}_r}\circ\psi_{r\#}.$ Let $x=x_1\otimes\cdots\otimes x_r\in H(M^{\times})^{\otimes r}$ be homogeneous. We have
\[
\begin{array}{lll}
\varphi_{M^r}(\eta(x))
&=&\varphi_{M^r}(x_1\times\cdots\times x_r)\\
&=&\pm\pi_{\#}((\tau_M\times\cdots\times\tau_M)\smallfrown(x_1\times\cdots\times x_r))\\
&=&\pm(-1)^{A_r(x)}\pi_{\#}((\tau_M\smallfrown x_1)\times\cdots\times(\tau_M\smallfrown x_r))\\
&=&\pm(-1)^{A_r(x)}(\varphi_M(x_1)\times\cdots\times\varphi_M(x_r))\\
&=&\pm\eta'\circ\overline{\varphi}\circ\kappa_M(x),
\end{array}
\]
where $A_r(x)$ is given by Lemma \ref{cap-product}. Therefore $\varphi_{M^r}\circ\eta=\pm\eta'\circ\overline{\varphi}\circ\kappa_M.$
\end{proof}

\begin{theorem}\label{mainfactor}
We have:
\[
\Lambda_{f_1\cdots f_p}=\Lambda_{FG}
=\pm\eta'\circ(\Lambda_{f_1f_2}\otimes\cdots\otimes\Lambda_{f_1f_p})\circ\kappa_X\circ\Delta_X^{(p-1)}.
\]
Equivalently, if
\[
\Delta_X^{(p-1)}(z)=\sum_{\nu}z_{\nu,1}\otimes\cdots\otimes z_{\nu,p-1}
\]
with homogeneous terms, then
\[
\Lambda_{f_1\cdots f_p}(z)
=\pm\sum_{\nu}(-1)^{A_{p-1}(z_\nu)}
\Lambda_{f_1f_2}(z_{\nu,1})\times\cdots\times\Lambda_{f_1f_p}(z_{\nu,p-1}).
\]
\end{theorem}

\begin{proof}
By Theorems \ref{factor-I}, \ref{lefschetz-p} and Proposition \ref{thom-kunneth} we obtain:
\[
\begin{array}{lll}
\Lambda_{f_1\cdots f_p}=\Lambda_{FG}
&=&\varphi_{M^{\times}_{p-1}}\circ I_{FG}\\
&=&\varphi_{M^{\times}_{p-1}}\circ\psi_{\#}\circ\eta\circ(I_{f_1f_2}\otimes\cdots\otimes I_{f_1f_p})\circ\Delta_X^{(p-1)}\\
&=&\varphi_{M^{p-1}}\circ\eta\circ(I_{f_1f_2}\otimes\cdots\otimes I_{f_1f_p})\circ\Delta_X^{(p-1)}\\
&=&\pm\eta'\circ\overline{\varphi}\circ\kappa_M\circ(I_{f_1f_2}\otimes\cdots\otimes I_{f_1f_p})\circ\Delta_X^{(p-1)}\\
&=&\pm\eta'\circ(\Lambda_{f_1f_2}\otimes\cdots\otimes\Lambda_{f_1f_p})\circ\kappa_X\circ\Delta_X^{(p-1)}.
\end{array}
\]
The last equality follows because each $I_{f_1f_j}$ has degree $0$ and $\varphi_M\circ I_{f_1f_j}=\Lambda_{f_1f_j}.$
\end{proof}

\begin{remark}
Similarly \cite{Saveliev}, is possible to define the coincidence and Lefschetz homomophisms for $p$-maps in the case where $M$ have boundary to study coincidences of maps $f_1,\cdots,f_p$ in the following situation; $f_1:(X,A)\to(M,\partial M)$ and $f_2,\cdots,f_p:X\to M.$
\end{remark}

\subsection{A cohomological formula and the image of the homomorphism}

Let $r=p-1,$ $N=mr,$ and let $d_r:M\to M^r$ be the diagonal map. Denote by
\[
J_r:M^r\times M^r\longrightarrow
(M^r\times M^r,M^r\times M^r\setminus\Delta_{M^r})
\]
the natural map to the pair, and put
\[
u_r=J_r^{\#}(\tau_{M^{\times}_r})
\in H^N(M^r\times M^r).
\]
Thus $u_r$ is the absolute cohomology class obtained from the Thom class used in Theorem \ref{lefschetz-p}. Define
\[
\mathcal L(f_1,\cdots,f_p)=(F,G)^{\#}(u_r)\in H^N(X),
\]
where $(F,G)(x)=(F(x),G(x)).$ These definitions use the fixed Thom orientation of that theorem.

\begin{proposition}[Cohomological formula]\label{cohomological-formula}
For every $z\in H_j(X)$ we have
\[
\Lambda_{f_1\cdots f_p}(z)
=F_{\#}\bigl(\mathcal L(f_1,\cdots,f_p)\smallfrown z\bigr)
=(d_r)_{\#}(f_1)_{\#}
\bigl(\mathcal L(f_1,\cdots,f_p)\smallfrown z\bigr).
\]
\end{proposition}

\begin{proof}
Write $b=(F,G):X\to M^r\times M^r.$ By definition,
$I_{FG}(z)=(J_r)_{\#}b_{\#}(z).$ Naturality of the cap product, first for the map to the pair and then for $b,$ gives
\[
\begin{aligned}
\Lambda_{f_1\cdots f_p}(z)
&={\pi'}_{\#}\bigl(\tau_{M^{\times}_r}
 \smallfrown (J_r)_{\#}b_{\#}(z)\bigr)\\
&={\pi'}_{\#}\bigl(u_r\smallfrown b_{\#}(z)\bigr)\\
&={\pi'}_{\#}b_{\#}
 \bigl(b^{\#}(u_r)\smallfrown z\bigr)\\
&=F_{\#}\bigl(\mathcal L(f_1,\cdots,f_p)\smallfrown z\bigr),
\end{aligned}
\]
where the first equality follows from Theorem \ref{lefschetz-p}. The last equality uses $\pi'\circ b=F.$ Since $F=d_r\circ f_1,$ the second formula follows.
\end{proof}

\begin{corollary}\label{image-diagonal}
For every $j,$
\[
\operatorname{Im}\bigl(\Lambda_{f_1\cdots f_p}|_{H_j(X)}\bigr)
\subset (d_r)_{\#}H_{j-mr}(M).
\]
In particular, $\Lambda_{f_1\cdots f_p}(z)=0$ if $j<mr$ or $j>mp.$
For $A=\emptyset,$ the same vanishing holds for $I_{FG}(z).$
\end{corollary}

\begin{proof}
The image assertion follows from Proposition \ref{cohomological-formula}. Since $H_k(M)=0$ for $k<0$ or $k>m,$ it also gives the stated vanishing of $\Lambda_{f_1\cdots f_p}.$ The vanishing of $I_{FG}$ follows because the Thom homomorphism in Theorem \ref{lefschetz-p} is an isomorphism.
\end{proof}

This improves the upper dimensional bound of Proposition 3.3, for $A=\emptyset$ and $p>2.$ Moreover, $(d_r)_{\#}$ is injective: if $\rho:M^r\to M$ is any coordinate projection, then $\rho\circ d_r=\Id_M.$ Consequently, the cohomological formula identifies the homomorphism with a class in the homology of $M,$ followed by the diagonal inclusion.

For completeness, the class in Proposition \ref{cohomological-formula} can also be written in terms of pairwise classes. Let $J_1:M\times M\to M^{\times}$ be the natural map to the pair and put
\[
u_1=J_1^{\#}(\tau_M),\qquad
\ell_{1j}=(f_1,f_j)^{\#}(u_1)\in H^m(X),
\quad j=2,\cdots,p.
\]
Then
\[
\mathcal L(f_1,\cdots,f_p)
=\varepsilon_r\,\ell_{12}\smile\ell_{13}\smile\cdots\smile\ell_{1p},
\qquad \varepsilon_r\in\{1,-1\}.
\]
Here $\varepsilon_r$ is the fixed orientation sign in Proposition \ref{thom-product}, not a sign depending on $z$ or on the maps. Indeed, pull back the product Thom class in that proposition along
\[
x\longmapsto ((f_1(x),f_2(x)),\cdots,(f_1(x),f_p(x)))
\]
and use the identity between an external product pulled back by a diagonal and a cup product. Equivalently, if $h=(f_1,\cdots,f_p):X\to M^p$ and
\[
e_r(x_1,\cdots,x_p)=((x_1,\cdots,x_1),(x_2,\cdots,x_p)),
\]
then $\mathcal L(f_1,\cdots,f_p)=h^{\#}(e_r^{\#}u_r).$ The class $e_r^{\#}u_r$ is, up to the fixed orientation sign, the absolute Thom class of the small diagonal $\Delta_p\subset M^p.$ To see this, in local coordinates its normal coordinates are
$(x_2-x_1,\cdots,x_p-x_1),$ which identify the normal bundle with the pullback of the normal bundle of $\Delta_{M^r}.$ Thus the formula relates the homological invariant to the cohomological coincidence class.

If $z\in H_N(X)$, the formula reduces to
\[
\Lambda_{f_1\cdots f_p}(z)
=\langle\mathcal L(f_1,\cdots,f_p),z\rangle[pt]
\in H_0(M^r).
\]
In higher degrees, the homomorphism also records the image under $f_1$ of the cap-product class.

For example, let $X=M^p$ and let $f_i=\pi_i:M^p\to M$ be the coordinate projections. The cap product of the absolute Thom class of $\Delta_p$ with $O_{M^p}$ is, up to orientation, the fundamental class of $\Delta_p.$ Since $\pi_1$ restricts to a homeomorphism $\Delta_p\to M,$ Proposition \ref{cohomological-formula} gives
\[
\Lambda_{\pi_1\cdots\pi_p}(O_{M^p})
=\pm(d_r)_{\#}(O_M)\neq 0\in H_m(M^r).
\]
The last inequality follows from the injectivity of $(d_r)_{\#}.$ This gives a non trivial value in positive degree when $m>0.$


\section{Examples}

{\bf (I)} Let $r=p-1,$ $X=M^r,$ fix a point $a\in M$ and consider the maps
\[
f_1(x_1,\cdots,x_r)=a,\qquad f_j(x_1,\cdots,x_r)=x_{j-1},\quad j=2,\cdots,p.
\]
Then $F:X\to M^r$ is the constant map with value $(a,\cdots,a)$ and $G:X\to M^r$ is the identity map. Thus
\[
\Coin(f_1,\cdots,f_p)=\{(a,\cdots,a)\}.
\]
Let $O_X\in H_{mr}(X)$ be the fundamental class. Since $F$ is constant, the transfer $F_!^{O_X}:H(X)\to H(X)$ is zero in every degree except $mr,$ and on the fundamental class it is the identity up to the choice of orientation. Therefore
\[
\Lambda_{f_1\cdots f_p}(O_X)=\pm[pt]\neq 0\in H_0(M^r).
\]
This gives a nontrivial Lefschetz homomorphism for every $p\geq 2.$

\

{\bf (II)} In this example we identify $H_0(M;\mathbb Q)$ with $\mathbb Q$ by sending $[pt]$ to $1$. Let $f,g:X\to M$ be maps where $X=M\times S^1$ and $M$ is an orientable connected suitable manifold of positive dimension $m$. For the definition of suitable, see \cite[Section 4]{Wong}. We define $\psi_{fg}:X\to M$ by
\[
\psi_{fg}(x)=g(x).[f(x)]^{-1}.
\]
We also define $\overline f,\overline g:M\to M$ by $\overline f(x)=f(x,1)$ and $\overline g(x)=g(x,1),$ where $1$ denotes the point $(1,0)\in S^1.$ If $\overline{O_M}$ is the dual of the fundamental class $O_M$ of $M$ then by \cite[Theorem 7.1]{Saveliev} we obtain:
\[
\Lambda_{fg}(z)=\langle\overline{O_M},{\psi_{fg}}_{\#}(z)\rangle
\]
for $z\in H_m(X).$ Let $j:H_m(M)\to H_m(M\times S^1)$ be given by $j(w)=w\otimes 1,$ where we are using the K$\ddot{u}$nneth isomorphism to identify
\[
H_m(M\times S^1)\approx\bigoplus_{k+l=m}H_k(M)\otimes H_l(S^1),
\]
and define $\overline{\psi_{fg}}:M\to M$ by $\overline{\psi_{fg}}(x)=\overline g(x).[\overline f(x)]^{-1}.$ Then $(\overline{\psi_{fg}})_{\#}={\psi_{fg}}_{\#}\circ j.$

From \cite[Theorem 3]{Wong} we have $\deg(\overline{\psi_{fg}})=L(\overline f,\overline g).$ Thus, considering $z=w\otimes 1,$ where $w\in H_m(M)$, we have
\[
\begin{array}{lll}
\Lambda_{fg}(z)
&=&\langle\overline{O_M},{\psi_{fg}}_{\#}(w\otimes 1)\rangle\\
&=&\langle\overline{O_M},(\overline{\psi_{fg}})_{\#}(w)\rangle\\
&=&\langle\overline{O_M},w\rangle L(\overline f,\overline g).
\end{array}
\tag{7}
\]
Therefore, if $w=O_M$ we obtain:
\[
\Lambda_{fg}(O_M\otimes 1)=L(\overline f,\overline g).
\tag{8}
\]

Equation (8) gives a calculation for two maps. For $p$-maps $f_1,\cdots,f_p:M\times S^1\to M$ with $p\geq 3,$ the element $O_M\otimes 1$ has degree $m<m(p-1),$ and therefore
\[
\Lambda_{f_1\cdots f_p}(O_M\otimes 1)=0.
\]
Thus the nontriviality of the pairwise numbers $L(\overline{f_1},\overline{f_j})$ does not by itself imply the nontriviality of $\Lambda_{f_1\cdots f_p}$ on the same element. The coproduct in Theorem \ref{mainfactor} has to be taken into account.

More generally, let $M$ be a closed connected oriented $m$-manifold, $m>0,$ admitting a fixed point free homeomorphism $R:M\to M.$ Let $h:X\to M$ be a continuous map such that $h_{\#}(z)=O_M$ for some $z\in H_m(X).$ Fix $a\in M$ and put
\[
f_1(x)=a,\qquad f_2(x)=h(x),\qquad f_3(x)=R(h(x)).
\]
Since $f_1$ is constant, its transfer with respect to $z$ is zero except in degree $m.$ The trace formula then gives
\[
\Lambda_{f_1f_2}(z)=(-1)^m[pt]\neq 0,
\qquad
\Lambda_{f_1f_3}(z)=\pm[pt]\neq 0,
\]
where the second equality uses $R_{\#}(O_M)=\pm O_M.$ However, $f_2(x)=f_3(x)$ would give a fixed point of $R.$ Thus $\Coin(f_1,f_2,f_3)=\emptyset$ and $\Lambda_{f_1f_2f_3}=0.$ This shows that nontrivial pairwise homomorphisms need not imply a common coincidence.

\

{\bf (III)} Let $r=p-1,$ $N=mr,$ $m\geq 1,$ and consider arbitrary continuous maps
\[
f_i:T^N\longrightarrow T^m,\qquad i=1,\cdots,p,
\]
where $T^k=\mathbb R^k/\mathbb Z^k$ has its standard product orientation. With respect to the standard bases of $H_1(T^N;\mathbb Z)$ and $H_1(T^m;\mathbb Z),$ let $A_i\in\operatorname{Mat}_{m\times N}(\mathbb Z)$ be the matrix of $(f_i)_{\#}.$ Define the square $N\times N$ matrix
\[
B=
\begin{pmatrix}
A_2-A_1\\
A_3-A_1\\
\vdots\\
A_p-A_1
\end{pmatrix}.
\]
The fact that $B$ is square uses $\dim T^N=m(p-1).$

\begin{proposition}\label{torus-determinant}
With the Lefschetz trace convention of Definition 2.6,
\[
\Lambda_{f_1\cdots f_p}(O_{T^N})
=(-1)^N\det(B)[pt]\in H_0((T^m)^r).
\]
In particular, $\det(B)\neq 0$ implies
$\Coin(f_1,\cdots,f_p)\neq\emptyset,$ and the same conclusion holds for any maps homotopic to the $f_i.$
\end{proposition}

\begin{proof}
Identify $(T^m)^r$ with $T^N$ by ordering the coordinates first by the $r$ factors and then by the $m$ coordinates in each factor. Consider the difference map
\[
D:T^N\longrightarrow T^N,\qquad
D(x)=(f_2(x)-f_1(x),\cdots,f_p(x)-f_1(x)).
\]
The group operation on the torus is continuous, and its induced map on $H_1$ is addition. Therefore the matrix of $D_{\#}$ on $H_1$ is $B.$

Let $\omega_N\in H^N(T^N)$ satisfy
$\langle\omega_N,O_{T^N}\rangle=1,$ and let
\[
q:T^N\times T^N\longrightarrow T^N,\qquad q(a,b)=b-a.
\]
Since $q^{-1}(0)=\Delta_{T^N}$ and $q$ identifies the normal directions to the diagonal with the tangent space at $0,$ its pullback of the relative orientation class at $0$ is a Thom class of the diagonal. Hence the absolute Thom class $u_r$ used above has the form
\[
u_r=\sigma_N q^{\#}(\omega_N),\qquad \sigma_N\in\{1,-1\}.
\]
The sign is fixed by the trace convention. Indeed, take $F_0:T^N\to T^N$ constant and $G_0=\Id_{T^N}.$ The transfer $F_{0!}^{O_{T^N}}$ is zero except in degree $N,$ where it fixes the fundamental class. Definition 2.6, in degree $n=0,$ therefore gives
\[
\Lambda_{F_0G_0}(O_{T^N})=(-1)^N[pt].
\]
On the other hand, Proposition \ref{cohomological-formula} evaluates this same value as $\sigma_N[pt],$ since $q\circ(F_0,G_0)$ is a translation. Thus $\sigma_N=(-1)^N.$

For the given maps, $q\circ(F,G)=D.$ Consequently,
\[
\mathcal L(f_1,\cdots,f_p)=(-1)^N D^{\#}(\omega_N).
\]
The cohomology ring of $T^N$ is the exterior algebra on its degree-one generators $t_1,\cdots,t_N,$ with
$\omega_N$ a nonzero scalar multiple of $t_1\smile\cdots\smile t_N.$ Since
\[
D^{\#}(t_a)=\sum_{b=1}^N B_{ab}t_b,
\]
expanding the product gives
\[
D^{\#}(\omega_N)=\det(B)\omega_N.
\]
The evaluation formula following Proposition \ref{cohomological-formula} now yields the claimed result. Finally, the matrices $A_i$ and hence the computed value are invariant under homotopy, and Theorem \ref{lefschetz-p} gives the coincidence assertion.
\end{proof}

This calculation requires no linearity assumption on the maps. If the maps are affine,
\[
f_i([x])=[A_i x+a_i],\qquad a_i\in\mathbb R^m,
\]
then $D([x])=[Bx+b]$ for the corresponding vector $b\in\mathbb R^N.$ If $\det(B)\neq 0,$ this is a covering with $|\det(B)|$ sheets. Thus there are exactly $|\det(B)|$ common coincidence points for these affine maps. Each has the same local index, namely $(-1)^N\operatorname{sgn}(\det B)$ with the convention above, so their sum agrees with Proposition \ref{torus-determinant}. The exact count is asserted only for the affine maps; the existence conclusion holds for all maps in their homotopy classes.

For example, let $p=3,$ $m=1,$ and define
\[
f_1([x,y])=[0],\qquad
f_2([x,y])=[2x+y],\qquad
f_3([x,y])=[x+3y].
\]
Then
\[
B=\begin{pmatrix}2&1\\1&3\end{pmatrix},
\qquad \det(B)=5,
\qquad
\Lambda_{f_1f_2f_3}(O_{T^2})=5[pt]\neq 0.
\]
The five common coincidence points are
\[
\left[\frac{k}{5},\frac{3k}{5}\right],\qquad k=0,\cdots,4.
\]
Moreover, $\Lambda_{f_1\cdots f_p}$ vanishes on $H_j(T^N)$ for $j<N.$ Thus in this example the determinant determines the whole homomorphism: it is nontrivial if and only if $\det(B)\neq 0.$

\

{\bf (IV)} We compute a positive-degree value directly from the Knill-like trace in Remark \ref{knilllike}. As in the calculations of \cite[Example 2.3]{GNO}, we use a homology basis, retaining the sign conventions of \cite{Saveliev} adopted in Section 2.
Let $X=T^3$, $M=S^1$, $p=3$, and let $f_1,f_2,f_3$ be the coordinate projections. Thus
\[
F(x,y,z)=(x,x),\qquad G(x,y,z)=(y,z).
\]
Choose the product fundamental classes
$O_{T^3}=e_x\times e_y\times e_z$ and $O_{T^2}=a\times b$,
where $e_x,e_y,e_z$ and $a,b$ are the oriented coordinate generators.
Let $\alpha,\beta\in H^1(T^2)$ be dual to $a,b$, and let
$\theta_x\in H^1(T^3)$ be the class of the first coordinate. For
$D(u)=u\smallfrown O_{T^2}$, Lemma \ref{cap-product} gives
\[
\alpha\smallfrown O_{T^2}=b,\qquad
\beta\smallfrown O_{T^2}=-a,
\qquad D^{-1}(a)=-\beta,\quad D^{-1}(b)=\alpha.
\]
Since $F^{\#}\alpha=F^{\#}\beta=\theta_x$ and
$\theta_x\smallfrown O_{T^3}=e_y\times e_z$, the degree-one homomorphism
$h=G_{\#}\circ F_!^{O_{T^3}}$ satisfies
\[
h(a)=-a\times b,\qquad h(b)=a\times b.
\]
Moreover, $h$ vanishes on $H_0(T^2)$ because $F$ factors through $S^1$ and hence $F^{\#}:H^2(T^2)\to H^2(T^3)$ is zero. It also vanishes on $H_2(T^2)$ because $H_3(T^2)=0$.
Thus only $k=1$ contributes to the trace, with sign
$(-1)^{k(k+1)}=1$, and
\[
\begin{aligned}
\Lambda_{f_1f_2f_3}(O_{T^3})
&=L_1(h)=\alpha\smallfrown h(a)+\beta\smallfrown h(b)\\
&=-(a+b)=-d_{2\#}(O_{S^1})\neq 0\in H_1(T^2),
\end{aligned}
\]
where $d_2:S^1\to T^2$ is the diagonal. This computes, by the trace formula, the positive-degree value of the homomorphism in the coordinate-projection example of Section 4.1.

\end{document}